\documentclass[12pt]{amsart}%

\usepackage{amssymb}
\usepackage{amsmath}
\usepackage{amsthm}
\usepackage{amscd}
\usepackage[margin=1in]{geometry}
\usepackage{hyperref}

\theoremstyle{plain}
\newtheorem{theorem}{Theorem}[section]
\newtheorem{proposition}{Proposition}[section]
\newtheorem{corollary}{Corollary}[section]
\newtheorem{lemma}{Lemma}[section]

\theoremstyle{remark}

\newtheorem{remark}{Remark}[section]
\newtheorem{examples}{Examples}[section]
\newtheorem{assumption}{Assumption}[section]

\DeclareMathOperator{\supp}{supp}

\DeclareMathOperator{\lin}{span}
\DeclareMathOperator{\cpct}{Cap}

\begin{document}

\title{Energy measure closability for Dirichlet forms}
\author{Michael Hinz$^1$}
\address{Mathematisches Institut, Friedrich-Schiller-Universit\"at Jena, Ernst-Abbe-Platz 2, 07737, Germany and Department of Mathematics,
University of Connecticut,
Storrs, CT 06269-3009 USA}
\email{Michael.Hinz.1@uni-jena.de and Michael.Hinz@uconn.edu}
\thanks{$^1$Research supported in part by NSF grant DMS-0505622 and by the Alexander von Humboldt Foundation Feodor (Lynen Research Fellowship Program)}
\author{Alexander Teplyaev$^2$}
\address{Department of Mathematics, University of Connecticut, Storrs, CT 06269-3009 USA}
\email{Alexander.Teplyaev@uconn.edu}
\thanks{$^2$Research supported in part by NSF grant DMS-0505622}

\date{\today}

\begin{abstract}
We consider symmetric Dirichlet forms on locally compact and non-locally compact spaces and provide an elementary proof for their closability with respect to energy dominant measures. We also discuss how to use known potential theoretic results to furnish an alternative proof 
of this theorem. 
\tableofcontents
\end{abstract}
\maketitle

\section{Introduction and setup}\label{S:Intro} 

The article is concerned with the closability of symmetric Dirichlet forms under a change of measure to an energy dominant reference measure, i.e. a measure with respect to which energy densites exist. It is well known that the existence of energy densities with respect to the reference measure entails a number of desirable properties, see our discussion below and \cite[Chapter I]{BH91}. In classical cases this property is immediate, cf. Example \ref{Ex:classical}, and often it is just assumed, but for some Dirichlet forms on fractals is atypical. In these cases we can turn to energy dominant measures in order to obtain natural constructions and statements, \cite{Hino08, Hino10, Hino11, Ku89, N85, T08}. We have encountered the need to use energy dominant measures when dealing with fiberwise representations for spaces of $1$-forms respectively vector fields in \cite[Section 2]{HRT}, see also \cite{AR89, AR90, Eb99} in this context. However, except for a few special cases (e.g. for local resistance forms or transient local Dirichlet forms) related closability results were not established. Here we prove that any regular symmetric Dirichlet form (together with its core) is closable with respect to any energy dominant measure. This produces a new Dirichlet form that has energy densities. We also provide a version of this result that applies to Dirichlet forms on measure spaces and therefore in particular to forms on non-locally compact topological spaces. To continue this introduction we collect some facts and notions.

For a Dirichlet form $(\mathcal{E},\mathcal{F})$ on the space $L_2(X,\mu)$ of $\mu$-square integrable functions on a measure space $(X,\mathcal{X},\mu)$ the subspace $\mathcal{F}\cap L_\infty(X,\mu)$ of $\mathcal{F}$ is an algebra (with multiplication defined pointwise $\mu$-a.e), and in particular
\begin{equation}\label{E:pointwisemult}
\mathcal{E}(fg)^{1/2}\leq \mathcal{E}(f)^{1/2}\left\|g\right\|_{L_\infty(X,\mu)}+\mathcal{E}(g)^{1/2}\left\|f\right\|_{L_\infty(X,\mu)},
\end{equation}
\cite[Corollary 3.3.2]{BH91}. If $X$ is a locally compact Hausdorff space, $\mu$ a nonnegative Radon measure on $X$ with full support and $(\mathcal{E},\mathcal{F})$ a regular symmetric Dirichlet form on $L_2(X,\mu)$ with core $\mathcal{C}$, then we have $\mathcal{C}\subset \mathcal{F}\cap L_\infty(X,\mu)$, and for any $f\in \mathcal{F}\cap L_\infty(X,\mu)$ the expression
\[L_f(\varphi):=\mathcal{E}(\varphi f,f)-\frac12\mathcal{E}(f^2,\varphi), \ \ \varphi\in\mathcal{C},\]
extends uniquely to a bounded and nonnegative linear functional $L_f$ on $C_0(X)$ which may be written 
\[L_f(\varphi)=\int_X\varphi d\Gamma(f), \ \ \varphi\in C_0(X),\]
with a uniquely determined Radon measure $\Gamma(f)$ on $X$, the \emph{energy measure} of $f$. By approximation energy measures $\Gamma(f)$ can be defined for any $f\in\mathcal{F}$, and by polarization we can subsequently define mutual energy measures $\Gamma(f,g)$ for any $f,g\in\mathcal{F}$. For details see for instance \cite{FOT94, LJ78, Sturm94, Sturm95}.

Following Hino \cite[Definition 2.1]{Hino10} we call a Radon measure $m$ on $X$ with full support an \emph{energy dominant measure for $(\mathcal{E},\mathcal{F})$} if all energy measures $\Gamma(f)$, $f\in\mathcal{F}$, are absolutely continuous with respect to $m$. An energy dominant measure for $(\mathcal{E},\mathcal{F})$ is called \emph{minimal} if it is absolutely continuous with respect to any other energy dominant measure for $(\mathcal{E},\mathcal{F})$, \cite[Definition 2.1]{Hino10}. If $m$ is a minimal energy dominant measure for $(\mathcal{E},\mathcal{F})$, then for any Borel set $B\subset X$ we have $m(B)=0$ if and only if $\Gamma(f)(B)=0$ for all $f\in\mathcal{F}$, \cite[Lemma 2.4]{Hino10}. In some sense minimal energy dominant measures are optimally adapted reference measures for Dirichlet forms, in particular in view of Lemma \ref{L:changemeasure} and Theorem \ref{T:changeofmeasure} below.

If $m$ is energy dominant with respect to $(\mathcal{E},\mathcal{F})$ then the Radon-Nikodym density 
\begin{equation}\label{E:RNdensity}
\Gamma(f):=\frac{d\Gamma(f)}{dm}
\end{equation}
of $\Gamma(f)$ with respect to $m$ is in $L_1(X,m)$ for any $f\in\mathcal{F}$. To $\Gamma(f)$ we refer as the \emph{energy density} of $f$ with respect to $m$. If there exist a unique positive and continuous bilinear form $\Gamma:\mathcal{F}\times\mathcal{F}\to L_1(X,\mu)$ such that 
\[\mathcal{E}(\varphi f,f)-\frac12\mathcal{E}(f^2,\varphi)=\int_X\varphi \Gamma(f)d\mu \]
for all $\varphi, f\in \mathcal{F}\cap L_\infty(X,\mu)$, then we say that $(\mathcal{E},\mathcal{F})$ admits a \emph{carr\'e du champ} (or \emph{square field operator}), see \cite[Chapter I, Definition 4.1.2]{BH91}. The original reference measure $\mu$ itself is energy dominant for $(\mathcal{E},\mathcal{F})$ if and only if $(\mathcal{E},\mathcal{F})$ admits a carr\'e du champ. In many cases the original reference measure $\mu$ itself is energy dominant. Finite dimensional examples for Dirichlet forms having this property are for instance provided by elliptic operators with regular coefficients on Euclidean domains or on smooth manifolds.
\begin{examples}\label{Ex:classical}
The most classical case is given by the Dirichlet integral
\[\mathcal{E}(f):=\int_\Omega |\nabla f(x)|^2 dx,\ \ f\in\mathcal{F},\]
where $\mathcal{F}$ is taken to be the $\mathcal{E}_1$-closure of $C_0^\infty(\Omega)$ in $L_2(\Omega)$, $\Omega\subset \mathbb{R}^n$ being a smooth domain. In this case the energy measure of $f\in\mathcal{F}$ is given by 
\[\Gamma(f)(A)=\int_A|\nabla f(x)|^2 dx, \ \ \text{$A\subset \Omega$ Borel}, \]
and the energy density with respect to the $n$-dimensional Lebesgue measure (restricted to $\Omega)$ is $\Gamma(f)=|\nabla f|^2$.
For an arbitrary symmetric regular Dirichlet form $(\mathcal{E},\mathcal{F})$ the energy density $\Gamma(f)$ may therefore be seen as a generalization of the square of the gradient $|\nabla f|^2$ and the energy measure as its more general measure version.
\end{examples} 

An example for a Dirichlet form on an infinite dimensional (and therefore non-locally compact) space that admits a carr\'e du champ is the Ornstein-Uhlenbeck form on the Wiener space, cf. \cite[Chapter II]{BH91}. However, if $(\mathcal{E},\mathcal{F})$ is a local Dirichlet form on a fractal and $\mu$ is the natural Hausdorff measure, then it may happen that \emph{all energy measures are singular with respect to $\mu$}. This behaviour is typical for diffusions associated with resistance forms on p.c.f. self-similar fractals and in many cases also for diffusions on Sierpinski carpets, \cite{BBST, Hino03, Hino05}. Though the original reference measure $\mu$ may fail to be energy dominant, energy dominant measures always exist and can be constructed explicitely. The following lemma is a version of \cite[Lemmas 2.2, 2.3 and 2.4]{Hino10}, \cite[Lemma 2.2]{HRT} and \cite[Lemma 2.2]{N85}. 

\begin{lemma}\label{L:changemeasure}\mbox{}
Let $X$ be a locally compact Hausdorff space, $\mu$ a Radon measure on $X$ with full support and $(\mathcal{E},\mathcal{F})$ a regular symmetric Dirichlet form on $L_2(X,\mu)$. Then there exists a minimal energy dominant measure $m$ for $(\mathcal{E},\mathcal{F})$. It is unique up to mutual equivalence of energy dominant measures. We may choose $m$ to be finite. 
\end{lemma}

We use the customary notation $\mathcal{E}_1(f):=\mathcal{E}(f)+\left\|f\right\|^2_{L_2(X,\mu)}$, $f\in \mathcal{F}$. A standard proof of Lemma \ref{L:changemeasure} is to take a countable family $\left\lbrace f_n\right\rbrace_n$ of functions from $\mathcal{F}$ such that $0<\mathcal{E}(f)\leq 1$ and $\lin (\left\lbrace f_n\right\rbrace_n)$ is $\mathcal{E}_1$-dense in $\mathcal{F}$, and to consider
\begin{equation}\label{E:typicalenergydominant}
m:=\sum_n a_n \Gamma(f_n),
\end{equation}
where $(a_n)_n$ is a suitable sequence of positive real numbers. Detailed arguments can be found in \cite{Hino10, HRT, N85}.

If $m$ is energy dominant for $(\mathcal{E},\mathcal{F})$ then the restriction $(\mathcal{E},\mathcal{C})$ of the original Dirichlet form to its core $\mathcal{C}$ is a densely defined nonnegative definite bilinear form on $L_2(X,m)$ that enjoys the Markov property. To perform a full change of measure it remains to be shown that $(\mathcal{E},\mathcal{C})$ is closable in $L_2(X,m)$. In this paper we give an elementary proof of this fact, see Section \ref{S:closability}.

\begin{theorem}\label{T:changeofmeasure} 
Let $X$ be a locally compact Hausdorff space, $\mu$ a Radon measure on $X$ with full support and $(\mathcal{E},\mathcal{F})$ a regular symmetric Dirichlet form on $L_2(X,\mu)$ with core $\mathcal{C}$. If $m$ is a finite energy dominant measure for $(\mathcal{E},\mathcal{F})$ then
$(\mathcal{E},\mathcal{C})$ is closable in $L_2(X,m)$ and its closure $(\mathcal{E},\mathcal{F}^{(m)})$ is a regular symmetric Dirichlet form on $L_2(X,m)$. If $X$ is second countable then the result is true also for infinite $m$.
\end{theorem}

We briefly refer to this circumstance as \emph{energy measure closability}. To our knowledge this Theorem in its present form is new, see however the remarks below and Section \ref{S:alternative} for its connection to preceding and related results. Our proof does not use capacities or quasi-notions, but is based on the Markov property and the Beurling-Deny decomposition together with some uniform integrability arguments. 

We state Theorem \ref{T:changeofmeasure} in the context of regular Dirichlet forms on locally compact spaces in the sense of \cite{FOT94} mainly because in this set up the existence of energy measures and the validity of a suitable Beurling-Deny decomposition are well known. However, we would like to emphasize that if we assume the latter, our method becomes applicable to Dirichlet forms on measure spaces as investigated in \cite[Chapter I]{BH91}, and we can obtain a corresponding result on energy measure closability in the measurable case, Theorem \ref{T:measurablecase}. The precise assumptions and a description of the necessary modifications of our proof can be found in Section \ref{S:measurablecase}.

For the special cases of transient Dirichlet forms on locally compact separable metric spaces and Dirichlet forms induced by resistance forms versions of Theorem \ref{T:changeofmeasure} had already been provided in \cite{HRT}, based on \cite{Ki12} and \cite{RW}, respectively. General results on the closability of Dirichlet forms under a change of measure can be found in \cite{AR90, BrKarw, FLJ, FST, Kuw92, KN91, Loeb03, RW}, see also the references cited in these papers. With Theorem \ref{T:changeofmeasure} in mind \cite{FLJ, FST} and \cite{KN91} are maybe most relevant. In Section \ref{S:alternative} we provide an alternative proof of energy measure closability which uses these references together with a reasoning similar to that of \cite[Lemma 5.1]{HRT}. 

To mention a particular consequence of energy measure closability let $(H, dom\:H)$ and $(H^{(m)}, dom\:H^{(m)})$ denote the infinitesimal generators of the Dirichlet forms $(\mathcal{E},\mathcal{F})$ and $(\mathcal{E},\mathcal{F}^{(m)})$, respectively. We write $(H^{1}, dom\:H^{1})$ for the smallest closed extension in $L_1(X,\mu)$ of the restriction of the operator $H$ to 
\[\left\lbrace f\in dom\:H\cap L_1(X,\mu): Hf\in L_1(X,\mu)\right\rbrace\]
as done in \cite[Proposition 2.4.2]{BH91}, and $(H^{(m),1}, dom\:H^{(m),1})$ if $m$ is used in place of $\mu$. Recall that $dom\:H^1\cap L_\infty(X,\mu)$ is an algebra if and only if $\mu$ is energy dominant, cf. \cite[Theorems 4.2.1 and 4.2.2]{BH91}. However, the stability under multiplication of a dense subspace of the domain of the Markov generator is desirable, because it permits a flexible use of finite energy coordinates, see for instance \cite{HRT, HTb, T08}. It seems that the choice of a suitable volume measure $m$ may be needed in order to be able to work with a particular coordinate system. We fix an immediate consequence of Theorem \ref{T:changeofmeasure} together with \cite[Theorems 4.2.1 and 4.2.2]{BH91}.

\begin{corollary}
Let the hypotheses of Theorem \ref{T:changeofmeasure} be in force. Then $dom\:H^{(m),1}\cap L_\infty(X,m)$ is an algebra, and for all $f\in dom\:H^{(m)}$ we have $f^2\in dom\:H^{(m),1}$.
\end{corollary}

In \cite{HKT} we have discussed a local-compactification method for Dirichlet forms. Any Dirichlet form $(\mathcal{E},\mathcal{F})$ on a measurable space can be transferred into a regular Dirichlet form $(\hat{\mathcal{E}},\hat{\mathcal{F}})$ on the Gelfand spectrum $\Delta$ of the uniform complex closure of an algebra of bounded energy finite functions. By Lemma \ref{L:changemeasure} we can always find a finite energy dominant measure $m$ for the transferred Dirichlet form $(\hat{\mathcal{E}},\hat{\mathcal{F}})$, and as a consequence of Theorem \ref{T:changeofmeasure}, the Dirichlet form $(\hat{\mathcal{E}},\hat{\mathcal{F}}^{(m)})$ obtained by change of measure always admits a carr\'e operator with respect to $m$. More details are provided in Section \ref{S:transfer}.

\begin{remark}
Theorem \ref{T:changeofmeasure} can for instance be used to generalize our previous results on the essential self-adjointness of magnetic Schr\"odinger Hamiltonians on fractals. In \cite{HTb} we have studied this issue for Dirichlet forms induced by local regular resistance forms \cite{Ki03, Ki12, T08}. Theorem \ref{T:changeofmeasure} allows to immediately extend the method of \cite{HTb} to regular symmetric Dirichlet forms on fractals that are not necessarily induced by resistance forms, such as for example local regular Dirichlet forms on generalized Sierpinski carpets with spectral dimension $d_S\geq 2$. Some background can be found in \cite{Ba, BB99, BBKT}.
\end{remark}

\section{Proof of energy measure closability}\label{S:closability}

In this section we prove Theorem \ref{T:changeofmeasure}. As before, let $(\mathcal{E},\mathcal{F})$ be a symmetric regular Dirichlet form on $L_2(X,\mu)$.

We first recall the Beurling-Deny decomposition, \cite[Theorem 3.2.1]{FOT94}, which itself may be seen as a consequence of the Markov property, see \cite[Th\'eor\`{e}me 1]{Allain}, \cite[Theorem 1.1 and Theorem 2.3]{Andersson}. Given $f\in\mathcal{C}$ we have
\begin{equation}
\mathcal{E}(f)=\mathcal{E}_c(f)+\mathcal{E}_j(f)+\mathcal{E}_k(f),
\end{equation}
where 
\[\mathcal{E}_c(f)=\int_Xd\Gamma_c(f)\]
is a symmetric strongly local bilinear form with domain $\mathcal{C}$ and $\Gamma_c(f)$ is the local part \cite[p. 113]{FOT94} of the energy measure $\Gamma(f)$ of $f$ (with a normalization different from that in \cite{FOT94}), 
\[\mathcal{E}_j(f)=\int_Xd\Gamma_j(f)\]
is a purely non-local form with 
\[\Gamma_j(f)(dx):=\int_X (f(x)-f(y)^2J(dxdy),\]
where $J$ is a symmetric nonnegative Radon measure on $X\times X\setminus \left\lbrace (x,x):x\in X\right\rbrace$, and
\[\mathcal{E}_k(f)=\int_Xd\Gamma_k(f)\]
is the killing part of $\mathcal{E}(f)$, where 
\[\Gamma_k(f)(dx):=f(x)^2\kappa(dx)\]
with a nonnegative Radon measure $\kappa$ on $X$. As a consequence we see that 
\begin{equation}\label{E:splitintegral}
\int_X\varphi\:d\Gamma(f)=\int_X \varphi\:d\Gamma_c(f)+\int_X\varphi\:d\Gamma_j(f)+\frac{1}{2}\int_X \varphi\:d\Gamma_k(f)
\end{equation}
for any $\varphi\in \mathcal{C}$, cf. \cite[page 114]{FOT94}, and as the measures $\Gamma_c(f)$, $\Gamma_j(f)$ and $\Gamma_k(f)$ are finite, equality (\ref{E:splitintegral}) extends to any bounded Borel function $\varphi$ on $X$ by bounded convergence. This also shows that
\begin{equation}\label{E:BDformeasures}
\Gamma(f)=\Gamma_c(f)+\Gamma_j(f)+\frac{1}{2}\Gamma_k(f),
\end{equation}
seen as an equality of nonnegative Radon measures on $X$. In particular, $\mathcal{E}(f)=\Gamma(f)(X)+\frac{1}{2}\mathcal{E}_k(f)$. For further details see \cite[Sections 3.2 and 5.3]{FOT94} or \cite[Section 2]{N85}. 

We will now record a simple consequence of the Markov property. For any $\alpha>0$ set
\begin{equation}\label{E:Talpha}
T_{\alpha}(y):=\begin{cases}
-\alpha\ &\text{ if $y\leq -\alpha$}\\
y\ &\text{ if $-\alpha < y <\alpha$}\\
\alpha\ &\text{ if $\alpha\leq y$.}\end{cases}
\end{equation}
Clearly each $T_{\alpha}$ is a Lipschitz function with $T_{\alpha}(0)=0$ and therefore $f\in\mathcal{C}$ implies $T_{\alpha}(f)\in\mathcal{C}$. 

\begin{lemma}\label{L:dominate}
There is a constant $c>0$ such that for any $f\in\mathcal{C}$ and any $\alpha>0$ we have
\begin{equation}\label{E:dominate}
\mathcal{E}(f-T_\alpha(f))\leq c\:\Gamma(f)(\left\lbrace |f|\geq \alpha\right\rbrace).
\end{equation}
\end{lemma}

For the local part $\mathcal{E}_c$ we observe the stronger result
\begin{equation}\label{E:stronger}
\mathcal{E}_c(f-T_\alpha(f))=\Gamma_c(f)(\left\lbrace |f|\geq \alpha\right\rbrace).
\end{equation}

\begin{proof} Let $U$ be a relatively compact open neighborhood of $\supp\:f$. Then also $U\cap \left\lbrace |f|<\alpha\right\rbrace$ is open and relatively compact. On this set we have $f-T_\alpha(f)=0$ by (\ref{E:Talpha}), what implies 
\[\mathcal{E}_c(f-T_\alpha(f))=\Gamma_c(f-T_\alpha(f))(\left\lbrace |f|\geq \alpha\right\rbrace),\]
by the locality of $\Gamma_c$, see for instance \cite[Corollary 3.2.1]{FOT94}. By the same reason $\Gamma_c(\alpha)(U)=0$, and with Cauchy-Schwarz 
the identity (\ref{E:stronger}) follows. For the non-local part note that by the contraction properties of $T_\alpha$,
\[I(x,y):=(f(x)-T_\alpha(f)(x)-f(y)+T_\alpha(f)(y))^2\leq 4(f(x)-f(y))^2\]
for any $x,y\in X$, and by (\ref{E:Talpha}),
\[\mathcal{E}_j(f-T_\alpha(f))=\int_{\left\lbrace |f|\geq \alpha\right\rbrace}\int_X I(x,y)J(dxdy)+\int_{\left\lbrace |f|< \alpha\right\rbrace}\int_{\left\lbrace |f|\geq \alpha\right\rbrace} I(x,y)J(dxdy).\]
For the second summand we may use the symmetry of $J$ to see it is bounded by the first. Consequently
\[\mathcal{E}_j(f-T_\alpha(f))\leq 8\Gamma_j(f)(\left\lbrace |f|\geq \alpha\right\rbrace).\]
For the killing part we have
\[\mathcal{E}_k(f-T_\alpha)\leq 4 \int_{\left\lbrace |f|\geq \alpha\right\rbrace}f^2d\kappa=4\Gamma_k(f)(\left\lbrace |f|\geq \alpha\right\rbrace),\]
and identity (\ref{E:BDformeasures}) finally implies (\ref{E:dominate}).
\end{proof}

Now recall that $m$ is an energy dominant measure for $(\mathcal{E},\mathcal{F})$. Let us agree to use notation (\ref{E:RNdensity}) with this very $m$. The next lemma is an observation concerning the uniform $m$-integrability of the energy densities of $\mathcal{E}$-Cauchy sequences. In a similar form this argument appeared already in \cite[Lemma 2.1]{Schmu92}, we sketch it for completeness.

\begin{lemma}\label{L:UI}
Let $(u_n)_n\subset \mathcal{F}$ be an $\mathcal{E}$-Cauchy sequence. Then the sequence $(\Gamma(u_n))_n$ is convergent in $L_1(X,m)$ and in particular, uniformly $m$-integrable, i.e. for any $\varepsilon>0$ we can find some $\delta>0$ such that for any Borel set $A\subset X$, $m(A)<\delta$ implies
\[\sup_n \Gamma(u_n)(A)<\varepsilon.\]
\end{lemma}
\begin{proof}
As shown in \cite[Lemma 2.5(i)]{Hino10} we have
\begin{equation}\label{E:bilinearpointwise}
|\Gamma(f)^{1/2}(x)-\Gamma(g)^{1/2}(x)|\leq \Gamma(f-g)^{1/2}(x)\ \ \text{for $m$-a.a. $x\in X$}
\end{equation}
and for any $f,g\in\mathcal{F}$. This follows easily because $\Gamma(sf+tg)\geq 0$ for all $s,t\in\mathbb{Q}$ $m$-a.e.
Reasoning similarly as in \cite[Lemma 2.1]{Schmu92} and \cite[Lemma 2.5(ii)]{Hino10},
\begin{align}
\int_X |\Gamma(u_m)-\Gamma(u_n)|dm & =\int_X |\Gamma(u_m)^{1/2}-\Gamma(u_n)^{1/2}|\left(\Gamma(u_m)^{1/2}+\Gamma(u_n)^{1/2}\right)dm \notag\\
&\leq 2\left(\int_X |\Gamma(u_m)^{1/2}-\Gamma(u_n)^{1/2}|^2dm\right)^{1/2}\:\sup_n\mathcal{E}(u_n)^{1/2},\notag
\end{align}
and integrating (\ref{E:bilinearpointwise}) we see that $\left(\Gamma(u_n)\right)_n\subset L_1(X,m)$ is Cauchy in $L_1(X,m)$, hence convergent, what implies uniform integrability.
\end{proof}

\begin{remark}
Alternatively we could use the integral form of (\ref{E:bilinearpointwise}), 
\[\left|\left(\int_X\varphi d\Gamma(f)\right)^{1/2}-\left(\int_X\varphi d\Gamma(g)\right)^{1/2}\right|\leq \left(\int_X\varphi d\Gamma(f-g)\right)^{1/2}\]
for $f,g\in \mathcal{F}$ and any bounded non-negative Borel function $\varphi$, see \cite[p. 111]{FOT94}, to conclude the uniform integrability from the Vitali-Hahn-Saks Theorem, cf. \cite[Theorem III.7.2]{DS}.
\end{remark}

For later use we record a particular consequence of Lemma \ref{L:UI}. It is a tightness result for the jump part.

\begin{corollary}\label{C:jumppart}
Assume that $m$ is finite or that $X$ is second countable. Let $(w_n)_n\subset\mathcal{F}$ be an $\mathcal{E}$-Cauchy sequence and let $\varepsilon>0$. Then there exists a compact set $K\subset X$ such that 
\begin{equation}\label{E:tightjump}
\sup_n\int_X\int_{X\setminus K}(w_n(x)-w_n(y))^2J(dxdy)<\varepsilon.
\end{equation}
\end{corollary}
\begin{proof}
Assume first that $m $ is finite. With the very $\varepsilon>0$ as given, let $\delta>0$ be as in Lemma \ref{L:UI}. Since $m$ is a finite Radon measure, we can find some compact $K\subset X$ with $m(X\setminus K)<\delta$, so that Lemma \ref{L:UI} entails (\ref{E:tightjump}).
Now assume $m$ is arbitrary but $X$ is second countable, hence Polish. Let $g$ be the limit in $L_1(X,m)$ of the sequence $(\Gamma(w_n))_n$. Then the sequence of measures $(\Gamma(w_n)\:dm)_n$ converges weakly on $X$ to the measure $gdm$, and as all subsequences inherit this convergence, Prohorov's theorem tells that $(\Gamma(w_n)\:dm)_n$ is tight. This implies (\ref{E:tightjump}).
\end{proof}

With the aid of Lemma \ref{L:UI} we can establish a key Proposition which allows to switch from a given $\mathcal{E}$-Cauchy sequence to a sequence that is decreasing in uniform norm.

\begin{proposition}\label{P:subseq}
Let $(u_n)_n\subset \mathcal{C}$ be a sequence that is $\mathcal{E}$-Cauchy and converges to zero in $L_2(X,m)$. Then there are a sequence $(k_j)_j\subset \mathbb{N}$ with $\lim_j k_j=\infty$ and a subsequence $(v_j)_j$ of $(u_n)_n$ such that
\begin{equation}\label{E:subCauchy}
\lim_j \mathcal{E}\left(v_j-T_{\frac{1}{k_j}}(v_j)\right)=0.
\end{equation}
\end{proposition}

\begin{proof}
According to Lemma \ref{L:UI} for any $j\in\mathbb{N}\setminus\left\lbrace 0\right\rbrace$ there exists some other $k_j$ such that for any $k\geq k_j$ 
\begin{equation}\label{E:UI}
m(A)<\frac{1}{k}\ \ \text{ implies } \ \ \sup_l\int_A\Gamma(u_{n_l})dm<\frac{1}{j}
\end{equation}
for any Borel set $A\subset X$.
As $(u_n)_n$ converges to zero in $L_2(X,m)$ we further observe that for any $k$ there is some $n_k$ such that for and $n\geq n_k$ we have
\begin{equation}\label{E:weak}
m\left(|u_n|\geq\frac{1}{k}\right)<\frac{1}{k}.
\end{equation}
Combining (\ref{E:UI}) and (\ref{E:weak}) shows that for any $n\geq n_{k_j}$ we have 
\[\sup_l\int_{\left\lbrace |u_n|\geq\frac{1}{k_j}\right\rbrace}\Gamma(u_{n_l})dm<\frac{1}{j}.\]
Writing $v_j:=u_{n_{k_j}}$ and using Lemma \ref{L:dominate} we see that in particular
\[\mathcal{E}\left(v_j-T_{\frac{1}{k_j}}(v_j)\right)\leq c\int_{\left\lbrace |v_j|\geq\frac{1}{k_j}\right\rbrace}\Gamma(v_j)dm<\frac{1}{j}.\]
\end{proof}

We finally prove Theorem \ref{T:changeofmeasure}.

\begin{proof}
Let $(u_n)_n\subset\mathcal{C}$ be a sequence that is $\mathcal{E}$-Cauchy and converges to zero in $L_2(X,m)$. Let $(v_j)_j$ be the subsequence of $(u_n)_n$ and $(k_j)_j$ the corresponding sequence of indices with (\ref{E:subCauchy}), shown to exist in Proposition \ref{P:subseq}. Clearly $(v_j)_j$ is $\mathcal{E}$-Cauchy, too. We have
\[\mathcal{E}\left(T_{\frac{1}{k_j}}(v_j)-T_{\frac{1}{k_l}}(v_l)\right)^{1/2}\leq \mathcal{E}\left(T_{\frac{1}{k_j}}(v_j)-v_j\right)^{1/2}+\mathcal{E}\left(T_{\frac{1}{k_l}}(v_l)-v_l\right)^{1/2}+\mathcal{E}(v_j-v_l)^{1/2},\]
which by (\ref{E:subCauchy}) is arbitrarily small, provided $j$ and $l$ are large enough. Consequently 
\[w_j:=T_{\frac{1}{k_j}}(v_j)\]
defines an $\mathcal{E}$-Cauchy sequence $(w_j)_j$. By construction
\[\sup_{x\in X}|w_j(x)|\leq \frac{1}{k_j}\]
for all $j$. Below we will show that 
\begin{equation}\label{E:neededtoshow}
\lim_j \mathcal{E}(w_j)=0.
\end{equation}
Then another application of (\ref{E:subCauchy}) yields $\lim_j\mathcal{E}(v_j)=0$, and since $(v_j)_j$ is an $\mathcal{E}$-convergent subsequence of the $\mathcal{E}$-Cauchy sequence $(u_n)_n$, we necessarily have
\[\lim_n \mathcal{E}(u_n)=0.\]
Because $\mathcal{C}$ is dense in $C_0(X)$ and therefore in $L_2(X,m)$, the bilinear form $(\mathcal{E},\mathcal{C})$ is a closable, densely defined and positive definite symmetric bilinear form on $L_2(X,m)$. It is easily seen to satisfy the Markov property. 

To verify (\ref{E:neededtoshow}), let $\varepsilon>0$. Choose $j_0=j_0(\varepsilon)$ such that for any $j,l\geq j_0$ we have 
\begin{equation}\label{E:Cauchyest}
\mathcal{E}(w_j-w_l)^{1/2}\sup_j\mathcal{E}(w_j)^{1/2}<\varepsilon.
\end{equation}
Let $K\subset X$ be the compact set which satisfied (\ref{E:tightjump}) with this very $\varepsilon$. We may assume it contains the closure $\overline{U}$ of a relatively compact open neighborhood $U$ of $\supp w_{j_0}$, otherwise consider $K\cup \overline{U}$ instead. By Urysohn's lemma for locally compact spaces and by the regularity of $(\mathcal{E},\mathcal{F})$ we can find a function $\varphi\in\mathcal{C}$ such that $0\leq \varphi\leq 1$ and $\varphi\equiv 1$ on $K$. By (\ref{E:pointwisemult}) the product $\varphi w_l$ is a function in $\mathcal{C}$ for each $l$, note also that $\supp \varphi w_l$ is contained in $\supp \varphi \cap \supp w_l$. The sequence $(\varphi w_l)_l$ is $\mathcal{E}$-Cauchy and converges to zero pointwise on $X$. By bounded convergence it therefore tends to zero in $L_2(X,\mu)$, and by the closability of $(\mathcal{E},\mathcal{C})$ in this space we obtain 
\[\lim_l \mathcal{E}(\varphi w_l)=0.\]
In particular, there exists some $l_0\geq j_0$ such that for all $j\geq l_0$ we have
\begin{equation}\label{E:fixbound}
\mathcal{E}(\varphi w_j)^{1/2}\sup_j\mathcal{E}(w_j)^{1/2}<\varepsilon.
\end{equation}
Let $j\geq l_0$ and note that 
\begin{equation}\label{E:sumenergies}
\mathcal{E}(w_j)=\mathcal{E}(w_{j_0}, w_{j_0}-w_j)+\mathcal{E}(w_{j_0}, w_j).
\end{equation}
By Cauchy-Schwarz and (\ref{E:Cauchyest}) the first summand in (\ref{E:sumenergies}) is bounded by $\varepsilon$. For the strongly local part of the second summand we have
\begin{equation}\label{E:stronglylocal}
\mathcal{E}_c(w_{j_0},w_j)=\int_X\varphi \Gamma_c(w_{j_0}, w_j)=\mathcal{E}_c(w_{j_0},\varphi w_j)<\varepsilon,
\end{equation}
where we have used the product rule, \cite[Lemma 3.2.5]{FOT94}, the (strong) local property of $\Gamma_c$, \cite[Corollary 3.2.1]{FOT94}, Cauchy-Schwarz and (\ref{E:fixbound}). The killing part can be estimated in a similar manner, 
\begin{equation}\label{E:killing}
\mathcal{E}_k(w_{j_0},w_j)=\mathcal{E}_k(\varphi w_{j_0},w_j)=\mathcal{E}_k(w_{j_0},\varphi w_j)<\varepsilon.
\end{equation}
For the jump part we have 
\begin{multline}
\mathcal{E}_j(w_{j_0},w_j)=\int_X\int_K(w_{j_0}(x)-w_{j_0}(y))(w_j(x)-w_j(y))J(dxdy)\notag\\
+\int_X\int_{X\setminus K}(w_{j_0}(x)-w_{j_0}(y))(w_j(x)-w_j(y))J(dxdy).
\end{multline}
By Cauchy-Schwarz and Corollary \ref{C:jumppart} the second summand is less or equal $\varepsilon$. The first can be rewritten as 
\begin{multline}\label{E:rewrite}
\int_K\int_K(w_{j_0}(x)-w_{j_0}(y))(\varphi(x)w_j(x)-\varphi(y)w_j(y))J(dxdy)\\
+\int_{X\setminus K}\int_K(w_{j_0}(x)-w_{j_0}(y))(w_j(x)-w_j(y))J(dxdy).
\end{multline}
The first double integral in (\ref{E:rewrite}) is bounded by $\mathcal{E}_j(w_{j_0})^ {1/2}\mathcal{E}_j(\varphi w_j)^ {1/2}<\varepsilon$. Due to the symmetry of $J$ the second double integral equals 
\[\int_K\int_{X\setminus K}  (w_{j_0}(x)-w_{j_0}(y))(w_j(x)-w_j(y))J(dxdy),\]
and using Cauchy-Schwarz and Corollary \ref{C:jumppart} yet another time, it is seen to be less than $\varepsilon$. Clipping the estimates, we obtain
\[\mathcal{E}(w_{j_0},w_j)< 6\:\varepsilon\]
for $j\geq l_0$. As $\varepsilon>0$ was arbitrary, this and (\ref{E:sumenergies}) show (\ref{E:neededtoshow}). Note that the preceding proof simplifies considerably if $\mu$ is a finite measure.
\end{proof}

\section{Generalization to the measurable case}\label{S:measurablecase}

In this section we discuss Dirichlet forms on measure spaces as in \cite[Chapter I]{BH91}. Under some additional assumptions we can generalize Theorem \ref{T:changeofmeasure} to this set up.

Let $(X,\mathcal{X},\mu)$ be a $\sigma$-finite measure space and $(\mathcal{E},\mathcal{F})$ a symmetric Dirichlet form on $L_2(X,\mu)$. Denote the space of all bounded $\mathcal{X}$-measurable functions by $b\mathcal{X}$. We consider
\begin{equation}\label{E:B}
\mathcal{B}:=\left\lbrace f\in b\mathcal{X}: \text{the $\mu$-equivalence class of $f$ is in $\mathcal{F}\cap L_1(X,\mu)$}\right\rbrace.
\end{equation}
By (\ref{E:pointwisemult}), Cauchy-Schwarz and the Markov property the space $\mathcal{B}$ is seen to be a multiplicative Stonean vector lattice. We write $\sigma(\mathcal{B})$ to denote the $\sigma$-ring of subsets of $X$ generated by $\mathcal{B}$, cf. \cite{Dudley}. We also use the notation 
\[(\mathcal{B}\otimes\mathcal{B})_d:=\left\lbrace \sum_i f_i\otimes g_i\in\mathcal{B}\otimes\mathcal{B}: (\sum_i f_i\otimes g_i)(x,x)=0\ \text{for all $x\in X$}\right\rbrace\]
and write $\sigma((\mathcal{B}\otimes\mathcal{B})_d)$ for the $\sigma$-ring of subsets of $X\times X$ generated by the space of functions $(\mathcal{B}\otimes\mathcal{B})_d$. The symbol $\mathcal{M}(\sigma(B))$ stands for the space of all finite signed measures on $\sigma(\mathcal{B})$. 

To obtain a version of Theorem \ref{T:changeofmeasure} we make the following additional assumptions:

\begin{assumption}\label{A:1}
We assume there exists a function $\chi\in\mathcal{B}$ that vanishes nowhere on $X$.
\end{assumption}

\begin{assumption}\label{A:2}
The Dirichlet form $(\mathcal{E},\mathcal{F})$ admits a Beurling-Deny decomposition in the following sense: There is a uniquely determined triplet $(\Gamma_c, J, \kappa)$ consisting of
\begin{enumerate}
\item[(i)] a map $\Gamma_c:\mathcal{B}\times\mathcal{B}\to \mathcal{M}(\sigma(\mathcal{B}))$ that satisfies $\Gamma_c(f)\geq 0$ for any $f\in\mathcal{B}$ and
\begin{equation}\label{E:productrule}
\Gamma_c(fg,h)=f\Gamma_c(g,h)+g\Gamma_c(f,h)\ \ \text{ for any $f,g,h\in\mathcal{B}$},
\end{equation}
\item[(ii)] a nonnegative measure $J$ on $\sigma((\mathcal{B}\otimes\mathcal{B})_d)$ with $(\mathcal{B}\otimes\mathcal{B})_d\subset L_1(X\times X, J)$ and 
\[\int_X\int_XF(x,y)J(dxdy)=\int_X\int_XF(y,x)J(dxdy)\ \ \text{ for all $F\in (\mathcal{B}\otimes\mathcal{B})_d$},\]
\item[(iii)] a nonnegative measure $\kappa$ on $\sigma(\mathcal{B})$  with $\mathcal{B}\subset L_1(X,\kappa)$ 
\end{enumerate}
such that 
for any $f,g\in\mathcal{B}$ we have
\[\mathcal{E}(f,g)=\mathcal{E}_c(f,g)+\mathcal{E}_j(f,g)+\mathcal{E}_k(f,g),\]
where the summands on the right hand side are the total masses on $X$ of the finite signed measures on $\sigma(\mathcal{B})$ given by $\Gamma_c(f,g)$,
\[\Gamma_j(f,g)(dx):=\int_X(f(x)-f(y))(g(x)-g(y))J(dxdy)\]
and
\[\Gamma_k(f,g)(dx):=f(x)g(x)\kappa(dx).\]
\end{assumption}

Assumption \ref{A:2} implies that $(\mathcal{E},\mathcal{F})$ \emph{admits energy measures}: For any $f\in\mathcal{B}$ there exists a finite nonnegative measure $\Gamma(f)$ on $\sigma(\mathcal{B})$ such that 
\[\mathcal{E}(fh,f)-\frac12\mathcal{E}(f^2,h)=\int_X hd\Gamma(f), \ \ h\in\mathcal{B}.\]
Mutual energy measures $\Gamma(f,g)$, for $f,g\in\mathcal{B}$ can then be constructed by polarization. A nonnegative measure $m$ on $\sigma(\mathcal{B})$ will be called \emph{energy dominant} if all $\Gamma(f)$, $f\in\mathcal{B}$, are absolutely continuous with respect to $m$.
We can always construct energy dominant measures for $(\mathcal{E},\mathcal{F})$ using the idea of Lemma \ref{L:changemeasure}, and we may choose them to be minimal and finite.

The following statement generalizes Theorem \ref{T:changeofmeasure}.

\begin{theorem}\label{T:measurablecase}
Let $(X,\mathcal{X},\mu)$ be a $\sigma$-finite measure space and $(\mathcal{E},\mathcal{F})$ a symmetric Dirichlet form on $L_2(X,\mu)$ satisfying Assumptions \ref{A:1} and \ref{A:2}. Let $m$ be a finite energy dominant measure for $(\mathcal{E},\mathcal{F})$. Then $(\mathcal{E},\mathcal{B})$ is closable in $L_2(X,m)$ and its closure $(\mathcal{E},\mathcal{F}^{(m)})$ is a symmetric Dirichlet form.
\end{theorem}

A proof follows by minor modifications of the arguments in Section \ref{S:closability}, we sketch them briefly. Note first that (\ref{E:productrule}) implies that $\Gamma_c$ is \emph{local} in the following sense: If $f\in\mathcal{B}$, $F:\mathbb{R}\to\mathbb{R}$ is Lipschitz with $F(0)=0$ and $F$ is constant on an interval $(a,b)\subset\mathbb{R}$, then 
\begin{equation}\label{E:generallocal}
\Gamma_c(F(f))\equiv 0\ \ \text{ on $\left\lbrace a<f<b\right\rbrace$}.
\end{equation}
This can be seen by a reasoning similar to \cite[Corollary 3.2.1]{FOT94}. Therefore the statement of Lemma \ref{L:dominate} holds for any $f\in\mathcal{B}$. Corresponding versions of Lemma \ref{L:UI} and Proposition \ref{P:subseq} hold anyway. Given an $\mathcal{E}$-Cauchy sequence $(u_n)_n$ that converges to zero in $L_2(X,m)$ we may proceed along the lines of the proof of Theorem \ref{T:changeofmeasure} in Section \ref{S:closability}, pass to a subsequence $(v_j)_j$ and consider the $\mathcal{E}$-Cauchy sequence $(w_j)_j$ with $w_j:=T_{\frac{1}{k_j}}(v_j)$ with $\sup_{x\in X}|w_j(x)|\leq 1/k_j$. Again we need to show (\ref{E:neededtoshow}), that is $\lim_j\mathcal{E}(w_j)=0$. If $\mu$ is finite this follows from the closability of $(\mathcal{E},\mathcal{F})$. To tackle the general case we use Assumption \ref{A:1}. Set $A_k:=\left\lbrace \chi>\frac{1}{k}\right\rbrace$, $k\in\mathbb{N}\setminus \left\lbrace 0\right\rbrace$. Then all $A_k$ have finite $\mu$-measure, $A_k\subset A_{k+1}$ for all $k$ and $X=\bigcup_k A_k$. By Lemma \ref{L:UI} and the finiteness and continuity of $m$ we can find some $k$ such that
\[\sup_j\Gamma(w_j)(X\setminus A_k)<\varepsilon.\]
The function 
\[\varphi_k:=k(k+1)(\chi\wedge \frac{1}{k}-\chi\wedge \frac{1}{k+1})\]
is a member of $\mathcal{B}$, equals $1$ on $A_k$ and vanishes outside $A_{k+1}$. Therefore $(\varphi_kw_j)_j$ converges to zero in $L_2(X,\mu)$ what implies
\[\lim_j \mathcal{E}(\varphi_k w_j)=0.\]
Now recall that given $\varepsilon>0$ we need to bound $\mathcal{E}(w_{j_0},w_j)$ in (\ref{E:sumenergies}). We have
\begin{equation}\label{E:referto}
\mathcal{E}(w_{j_0},w_j)=\int_{A_k}d\Gamma(w_{j_0}, w_j)+\int_{X\setminus A_k}d\Gamma(w_{j_0}, w_j)+\frac12\int_{A_k}d\Gamma_k(w_{j_0}, w_j)+\frac12\int_{X\setminus A_k}d\Gamma_k(w_{j_0}, w_j),
\end{equation}
and by Cauchy-Schwarz the second and fourth summand cannot exceed 
\[\sup_j \mathcal{E}(w_j)^{1/2}\sup_j\Gamma(w_j)(X\setminus A_k)^{1/2}.\] 
For the local part of the first summand in (\ref{E:referto}) we have 
\[\int_{A_k}d\Gamma_c(w_{j_0},w_j)=\int_{A_k}\varphi_k d\Gamma_c(w_{j_0},w_j)=\int_{A_k}d\Gamma(w_{j_0},\varphi_k w_j)\leq \mathcal{E}(\varphi_kw_j)^{1/2}\sup_j\mathcal{E}(w_j)^{1/2},\]
note that $A_k=\bigcup_{\delta>0}\left\lbrace \chi>\frac1k+\delta\right\rbrace$ and on each set $\left\lbrace \chi>\frac1k+\delta\right\rbrace$ we have $\Gamma_c(w_{j_0},\varphi_kw_j)=\Gamma_c(w_{j_0},w_j)$. The nonlocal part of the first summand in (\ref{E:referto}) can be rewritten and estimated in a similar manner as in (\ref{E:rewrite}), and a bound for the killing parts is straightforward. As before we may conclude that there is a constant $c>0$, independent of $j_0$, such that $\mathcal{E}(w_{j_0},w_j)<c\:\varepsilon$ for any sufficiently large $j$.

\section{Quasi-supports and an alternative proof}\label{S:alternative}

In this section we sketch an alternative proof of Theorem \ref{T:changeofmeasure}, this time based on potential theoretic arguments in \cite{FLJ, FST, KN91}. Throughout this section we adopt the hypotheses of these papers and therefore assume that $X$ is a locally compact compact separable metric space, $\mu$ a Radon measure on $X$ with full support and $(\mathcal{E},\mathcal{F})$ a regular symmetric Dirichlet form on $L_2(X,\mu)$ with core $\mathcal{C}$. 

The $1$-capacity associated with $(\mathcal{E},\mathcal{F})$ is defined as
\[\cpct(A):=\inf\left\lbrace \mathcal{E}_1(u): \text{$u\in\mathcal{F}$ and $u\geq 1$ $\mu$-a.e. on $A$}\right\rbrace \]
for open sets $A\subset X$. If the infimum is taken over the empty set, $\cpct(A)$ is set to be infinity. The $1$-capacity of an arbitrary subset $A\subset X$ is defined to be 
\[\cpct(A):=\inf\left\lbrace \cpct(B): \text{$B\supset A$, $B$ open}\right\rbrace.\]
If $(\mathcal{E},\mathcal{F})$ is transient we can also define the associated $0$-capacity using $\mathcal{E}$ in place of $\mathcal{E}_1$, it will be denoted by $\cpct_0$.

The following statement is an immediated consequence of \cite[Lemma 3.2.4]{FOT94} together with \cite[Lemma 2.3]{Hino10}, we just have to consider minimal energy dominant measures of type (\ref{E:typicalenergydominant}). 
\begin{theorem}
Let $m$ be a minimal energy dominant measure for $(\mathcal{E},\mathcal{F})$. Then $m$ charges no set of zero capacity.
\end{theorem}

On the other hand, if $m$ is energy dominating but not necessarily minimal, we can consider the part of $m$ that is absolutely continuous with respect to $\cpct$. Since $m$ is a Radon measure and $\cpct$ is countably subadditive, there exist Radon measures $m_0$ and $m_1$ and a Borel set $N\subset X$ of zero capacity such that $m$ decomposes uniquely into the sum
\begin{equation}\label{E:splitmeasure}
m=m_0+m_1,
\end{equation}
where $m_0=\mathbf{1}_{X\setminus N}m$ is absolutely continuous with respect to $\cpct$ and $m_1=\mathbf{1}_N m$. See for instance \cite[Lemma 2.1]{FST}.

\begin{theorem}\label{T:alsoenergydom}
Let $m$ be an energy dominant measure for $(\mathcal{E},\mathcal{F})$ and $m_0$ its absolutely continuous part with respect to $\cpct$ as in (\ref{E:splitmeasure}). Then also $m_0$ is energy dominant.
\end{theorem}
\begin{proof}
Let $f\in\mathcal{F}$. If $A\subset X$ is a Borel set such that $0=m_0(A)=m(A\setminus N)$, then also $\Gamma(f)(A\setminus N)=0$, because $m$ is energy dominant. Since $\Gamma(f)$ does not charge sets of zero capacity we have $\Gamma(f)(A\cap N)\leq \Gamma(f)(N)=0$, too.
\end{proof}

\begin{remark}
The measure $m_0$ is generally not minimal energy dominant. For example, consider
\[\mathcal{E}(f)=\int_0^1 f'(x)^2dx,\]
$\mu(dx)=dx$ being the Lebesgue measure on $(0,1)$ and $\mathcal{F}\subset L_2(0,1)$ the $\mathcal{E}_1$-closure of $C^1_0(0,1)$. Then $(\mathcal{E},\mathcal{F})$ is a regular Dirichlet form on $L_2(0,1)$. Let $\left\lbrace f_n\right\rbrace_n\subset C_0(0,1)\cap \mathcal{F}$ be a countable family of functions with $(f_n')^2\leq 1$ and $\lin(\left\lbrace f_n\right\rbrace_n)$ dense in $\mathcal{F}$. 
Let $\delta_{1/2}$ be the normalized Dirac measure at $1/2$. Then both 
\[m':=\sum_n 2^{-n} (f_n'(x))^2dx \ \ \text{ and }\ \ m:=m'+\delta_{1/2}\]
are energy dominant measures for $(\mathcal{E},\mathcal{F})$. Since points have positive capacity, both are absolutely continuous with respect to $\cpct$, and in particular $m_0=m$ in (\ref{E:splitmeasure}). But $m$ is not minimal among the energy dominant measures, because 
\[m(\left\lbrace 1/2\right\rbrace)=1\ \ \text{ and }\ \ m'(\left\lbrace 1/2\right\rbrace)=0.\]
\end{remark}

A set $E\subset X$ is \emph{quasi-open} if for any $\varepsilon >0$ there exists an open set $G$ containing $E$ such that $\cpct(G\setminus E)=0$. A set is said to be \emph{quasi-closed} if it is the complement of a quasi-open set. A function on $X$ is called \emph{quasi-continuous} if for any $\varepsilon>0$ there exists an open set $G\subset X$ with $\cpct(G)<\varepsilon$ and the function is continuous on $X\setminus G$. Any element $u\in\mathcal{F}$ has an $m$-version that is quasi-continuous. See \cite[Theorem 2.1.3]{FOT94}. We will denote this version by $\widetilde{u}$. If a property holds on $X\setminus N$, where $N\subset X$ is a set of zero capacity, $\cpct(N)=0$, then we say this property holds \emph{quasi-everywhere}, abbreviated \emph{q.e.} For $A,B\subset X$ we write $A\subset B$ q.e. if $\cpct(A\setminus B)=0$. Given a nonnegative Radon measure $\nu$ on $X$ that charges no set of zero capacity, a set $\widetilde{F}\subset X$ is called a \emph{quasi-support} for $\nu$ if $\widetilde{F}$ is quasi-closed, $\nu(X\setminus \widetilde{F})=0$ and for any other set $\check{F}\subset X$ with these properties we have $\widetilde{F}\subset \check{F}$ q.e. The measure $\nu$ is said to have \emph{full quasi-support} if $X$ itself is a quasi-support for $\nu$. The following condition is necessary and sufficient for $\nu$ to have full quasi-support:
\begin{equation}\label{E:condqs}
\text{$\widetilde{u}=0$ $\nu$-a.e. if and only if $\widetilde{u}=0$ for any $u\in\mathcal{F}$.}
\end{equation}
A proof of this equivalence is given in \cite[Theorem 3.3]{FLJ}. Here we are interested in the quasi-supports of energy dominant measures.

\begin{theorem}\label{T:fullqs}
Assume $(\mathcal{E},\mathcal{F})$ is irreducible or transient. Let $m$ be an energy dominant measure for $(\mathcal{E},\mathcal{F})$ that does not charge sets of zero capacity. Then $m$ has full quasi-support.
\end{theorem}

Irreducibility or transience are the standard assumptions for the results in \cite{FLJ, FST, KN91} that are relevant for our proof of Theorem \ref{T:alsoenergydom}. Note that they were not needed in Section \ref{S:closability}. The next corollary is an immediate consequence of Theorem \ref{T:alsoenergydom}.

\begin{corollary}
Assume $(\mathcal{E},\mathcal{F})$ is irreducible or transient. If $m$ is an energy dominant measure for $(\mathcal{E},\mathcal{F})$, then $m_0$ has full quasi-support.
\end{corollary}

Our main conclusion from Theorem \ref{T:fullqs} will be a special case of Theorem \ref{T:changeofmeasure}. It follows from Theorem \ref{T:fullqs} by the arguments of \cite[Section 5, in particular Theorem 5.3]{FLJ}.

\begin{theorem}
Let $(\mathcal{E},\mathcal{F})$ be a regular Dirichlet form on $L_2(X,\mu)$ with core $\mathcal{C}$. Assume $(\mathcal{E},\mathcal{F})$ is irreducible or transient and let $m$ be an energy dominant measure for $(\mathcal{E},\mathcal{F})$. Then $(\mathcal{E},\mathcal{C})$ is closable 
in $L_2(X,m)$.
\end{theorem}

To prove Theorem \ref{T:fullqs} we first establish the following lemma. It is a version of \cite[Theorem 5.2.3]{BH91}.
\begin{lemma}\label{L:zerolevel}
Let $u\in\mathcal{F}$ be such that $\widetilde{u}=0$ $\Gamma(u)$-a.e. Then also $\Gamma(u)(\left\lbrace \widetilde{u}=0\right\rbrace)=0$.
\end{lemma}

\begin{proof}
For the jump part of the energy measure we observe
\begin{align}
\Gamma_j(u)(\left\lbrace \widetilde{u}=0\right\rbrace)&=\int_{\left\lbrace \widetilde{u}=0\right\rbrace}\int_X (\widetilde{u}(x)-\widetilde{u}(y))^2 J(dxdy)\notag\\
&=\int_{\left\lbrace \widetilde{u}=0\right\rbrace}\int_{\left\lbrace \widetilde{u}\neq 0\right\rbrace}(\widetilde{u}(x)-\widetilde{u}(y))^2J(dxdy)\notag\\
&\leq \int_{\left\lbrace \widetilde{u}\neq 0\right\rbrace}\int_X (\widetilde{u}(x)-\widetilde{u}(y))^2J(dxdy)\notag\\
&=\Gamma_j(u)(\left\lbrace \widetilde{u}\neq 0\right\rbrace)\notag\\
&=0\notag
\end{align}
by the symmetry of $J$. The killing part also vanishes,
\[\Gamma_k(u)(\left\lbrace \widetilde{u}=0\right\rbrace)=\int_{\left\lbrace \widetilde{u}=0\right\rbrace} \widetilde{u}(x)^2\kappa(dx)=0.\]
For the strongly local part we follow the argument of \cite[Theorem 5.2.3]{BH91} and use the closability of $(\mathcal{E},\mathcal{F})$. Let $(\varphi_n)_n$ be a sequence of $C^1(\mathbb{R})$-functions with $0\leq \varphi\leq 1$ and compactly supported that approximate $\mathbf{1}_{\left\lbrace 0\right\rbrace}$ pointwise. Then we also have
\begin{equation}\label{E:pointwiseae}
\lim_n \varphi_n(\widetilde{u})=\mathbf{1}_{\left\lbrace 0\right\rbrace}(\widetilde{u}) \ \ \text{$\Gamma(u)$-a.e.}
\end{equation}
By (\ref{E:pointwiseae}) and bounded convergence, $(\varphi_n(\widetilde{u}))_n$ converges to $\mathbf{1}_{\left\lbrace 0\right\rbrace}(\widetilde{u})$ in $L_2(X, \Gamma_c(u))$. For any $n$ set
\[F_n(s):=\int_{0}^s \varphi_n(t)dt.\]
By the chain rule \cite[Theorem 3.2.2]{FOT94} it now follows that
\begin{equation}\label{E:localenergytozero}
\lim_n \mathcal{E}_c(F_n(u))=\lim_n \int_X\varphi_n(\widetilde{u})^2d\Gamma_c(u)=\Gamma_c(u)(\left\lbrace \widetilde{u}=0\right\rbrace),
\end{equation}
and as $(\varphi_n(\widetilde{u}))_n$ is Cauchy in $L_2(X,\Gamma_c(u))$, the sequence $(F_n(u))_n$ is $\mathcal{E}_c$-Cauchy, because
\[\mathcal{E}_c(F_n(u)-F_k(u))=\int_X(\varphi_n(\widetilde{u})-\varphi_k(\widetilde{u}))^2d\Gamma_c(u).\]
But it is also $\mathcal{E}_j$- and $\mathcal{E}_k$-Cauchy: We have
\[|F_n(\widetilde{u})(x)-F_n(\widetilde{u})(y)|=|\int_{\widetilde{u}(x)}^{\widetilde{u}(y)}\varphi_n(t)dt|=|\widetilde{u}(x)-\widetilde{u}(y)|\int_0^1\varphi_n((1-\lambda)\widetilde{u}(x)+\lambda \widetilde{u}(y))d\lambda,\]
which is bounded by $|\widetilde{u}(x)-\widetilde{u}(y)|$ and goes to zero by bounded convergence on $(0,1)$. Dominated convergence (w.r.t. $J$) therefore shows that
\[\lim_n \mathcal{E}_j(F_n(\widetilde{u}))=\lim_n\int_X\int_X(F_n(\widetilde{u})(x)-F_n(\widetilde{u})(y))^2J(dxdy)=0.\]
For the killing part we similarly observe 
\[\lim_n\mathcal{E}_k(F_n(u))=\lim_n\int_XF_n(\widetilde{u}(x)^2\kappa(dx)=\int_X\left(\int_{0}^{\widetilde{u}(x)}\varphi_n(t)dt\right)^2\kappa(dx)=0.\]
Therefore $(F_n(u))_n$ is $\mathcal{E}$-Cauchy. By dominated convergence $(F_n(u))_n$ goes to zero in $L_2(X,\mu)$. The closability of $(\mathcal{E},\mathcal{F})$ therefore implies 
\[\lim_n \mathcal{E}_c(F_n(u))\leq \lim_n \mathcal{E}(F_n(u))=0,\]
and together with (\ref{E:localenergytozero}), $\Gamma_c(\left\lbrace \widetilde{u}=0\right\rbrace)=0$.
\end{proof}

\begin{remark}\mbox{}
\begin{enumerate}
\item[(i)] If $(\mathcal{E},\mathcal{F})$ is local, the condition $\widetilde{u}=0$ $\Gamma(u)$-a.e. is not needed. 
\item[(ii)] In the proof of \cite[Theorem 5.2.3]{BH91} closability was used to show that for any strongly local Dirichlet form $(\mathcal{E},\mathcal{F})$ we have $\Gamma(u)(\left\lbrace \widetilde{u}\in K\right\rbrace)=0$ for any $u\in\mathcal{F}$ and any set $K\subset \mathbb{R}$ of zero Lebesgue measure. This fact implies the validity of the chain rule for Lipschitz transformations. 
\end{enumerate}
\end{remark}

We prove Theorem \ref{T:fullqs}.
 
\begin{proof} It suffices to check condition (\ref{E:condqs}). If $u\in\mathcal{F}$ is such that $\widetilde{u}=0$ q.e. then also $\widetilde{u}=0$ $m$-a.e. because $m$ does not charge sets of zero capacity. To verify the converse, let $u\in\mathcal{F}$ be such that $\widetilde{u}=0$ $m$-a.e. Then we have $\Gamma(u)(X)=0$ by Lemma \ref{L:zerolevel} and therefore $\mathcal{E}(u)=0$. Now consider the case that $(\mathcal{E},\mathcal{F})$ is irreducible. Following \cite{FST} and \cite{KN91} set
\[\mathcal{E}^{m}(f,g):=\mathcal{E}(f,g)+\left\langle f,g\right\rangle_{L_2(X,m)}\]
for $f,g\in \widetilde{\mathcal{F}}\cap L_2(X,m)$, where $\widetilde{\mathcal{F}}$ denotes the collection of all $\mathcal{E}$-quasi-continuous versions of elements of $\mathcal{F}$. Then $(\mathcal{E}^{m},\widetilde{\mathcal{F}}\cap L_2(X,m))$ is a Dirichlet form on $L_2(X,\mu)$, see \cite[Lemma 6.1.1]{FOT94} and obviously $\mathcal{C}\subset \widetilde{\mathcal{F}}\cap L_2(X,m)$. Moreover, by \cite[Theorem 2.1 and Proposition 2.2]{KN91} the Dirichlet form $(\mathcal{E}^{m},\widetilde{\mathcal{F}}\cap L_2(X,m))$ is regular and transient. Note that the $m$-equivalence class of the Borel function $\widetilde{u}$ has zero $L_2(X,m)$-norm and hence is a member of $\widetilde{\mathcal{F}}\cap L_2(X,m)$. Now let $\cpct^{m}$ and $\cpct_0^{m}$ denote the $1$- and $0$-capacity associated with $(\mathcal{E}^{m},\widetilde{\mathcal{F}}\cap L_2(X,m))$, respectively. From their definition it follows that $\cpct_1^{m}$ dominates $\cpct_1$. Hence any $\mathcal{E}^{m}$-quasi-continuous version of $\widetilde{u}$ is also an $\mathcal{E}$-quasi-continuous version of $u$. We may therefore assume that $\widetilde{u}$ is $\mathcal{E}^{m}$-quasi-continuous. Then for any $\varepsilon>0$ the weak capacitary inequality 
\begin{equation}\label{E:weakcapineq}
\cpct_0^{m}(\left\lbrace|\widetilde{u}|>\varepsilon\right\rbrace)\leq \frac{1}{\varepsilon^2} \mathcal{E}^{m}(\widetilde{u})
\end{equation}
holds, and since $\mathcal{E}^{m}(\widetilde{u})=\mathcal{E}(u)=0$, we may conclude that $\cpct_0^{m}(\left\lbrace|\widetilde{u}|>0\right\rbrace)=0$. By transience this implies $\cpct_1^{m}(\left\lbrace|\widetilde{u}|>0\right\rbrace)=0$ and therefore $\cpct_1(\left\lbrace|\widetilde{u}|>0\right\rbrace)=0$, that is $\widetilde{u}=0$ q.e. If $(\mathcal{E},\mathcal{F})$ is transient the proof simplifies, because (\ref{E:weakcapineq}) is valid for $\mathcal{E}$ itself and a set is of zero $1$-capacity if and only if it is of zero $0$-capacity.
\end{proof}

\section{Dirichlet forms on the Gelfand spectrum}\label{S:transfer}

In some situations the space $X$ itself may not possess nice topological properties, or the Dirichlet form $(\mathcal{E},\mathcal{F})$ may not be 
regular. In \cite{HKT} we have discussed a method to 'embed' an image of $X$ into a locally compact Hausdorff space $\Delta$ and to transfer the original Dirichlet form $(\mathcal{E},\mathcal{F})$ on $X$ into a regular Dirichlet form $(\hat{\mathcal{E}},\hat{\mathcal{F}})$ on $\Delta$. To be more precise, let $(X,\mathcal{X},\mu)$ be a measure space and, as in Section \ref{S:measurablecase}, let $b\mathcal{X}$ denote the space of all bounded $\mathcal{X}$-measurable functions on $X$. Let $(\mathcal{E},\mathcal{F})$ be a Dirichlet form on $L_2(X,\mu)$ and recall the definition (\ref{E:B}) of the algebra $\mathcal{B}$, that is
\[\mathcal{B}:=\left\lbrace f\in b\mathcal{X}: \text{the $\mu$-equivalence class of $f$ is in $\mathcal{F}\cap L_1(X,\mu)$}\right\rbrace.\]
We additionally assume that $\mathcal{B}$ vanishes nowhere on $X$, this assumption is weaker than Assumption \ref{A:1}. Now consider the complex closure $A(\mathcal{B})$ of $\mathcal{B}$ in the supremum norm. The algebra $A(\mathcal{B})$ is a $\mathcal{C}^\ast$-algebra, and together with the Gelfand topology the set $\Delta$ of its nonzero characters becomes a locally compact Hausdorff space, called the \emph{Gelfand spectrum} of $A(\mathcal{B})$, see e.g. \cite{Kaniuth}. It 'contains' $X$, more precisely, each $x\in X$ defines a nonzero homomorphism $\iota(x)\in\Delta$ from the $\mathcal{C}^\ast$-algebra into $\mathbb{C}$ by $\iota(x)(f):=f(x)$, and the image $\iota(X)$ is a dense subset of $\Delta$. Moreover, there is a uniquely determined nonnegative Radon measure $\hat{\mu}$ on $\Delta$ such that 
\[\int_X f d\mu=\int_{\Delta}\hat{f}d\hat{\mu}\ \ \text{ for all $f\in\mathcal{B}$},\]
where $f\mapsto \hat{f}$ denotes the Gelfand transform, \cite{Kaniuth}. Setting 
\[\hat{\mathcal{E}}(\hat{f},\hat{g}):=\mathcal{E}(f,g), \ \ f,g\in \mathcal{B},\]
we obtain a densely defined, non-negative definite symmetric bilinear form on $L_2(\Delta,\hat{\mu})$. In \cite[Theorem 5.1]{HKT} we have shown that $(\hat{\mathcal{E}},\hat{\mathcal{B}})$ is closable in $L_2(X,\hat{\mu})$, and its closure $(\hat{\mathcal{E}},\hat{\mathcal{F}})$ is a symmetric regular Dirichlet form. Here $\hat{\mathcal{B}}$ denotes the image of $\mathcal{B}$ under the Gelfand transform. We refer to $(\hat{\mathcal{E}},\hat{\mathcal{F}})$ as the \emph{transferred Dirichlet form}.

\begin{remark}
For spaces $X$ that carry a (non-locally compact) topology a similar procedure has been used by several authors to embed $X$ into a locally compact space, \cite{AR89, Ku82}, for instance to construct a Hunt process on $X$ associated with $(\mathcal{E},\mathcal{F})$. To ensure the embedding has sufficiently nice properties, some additional conditions must be imposed, cf. \cite[Section 2]{AR89}.
\end{remark}

This setup can be combined with Lemma \ref{L:changemeasure} and Theorem \ref{T:changeofmeasure}. We use the notation $\hat{\mathcal{B}}_c:=\left\lbrace \hat{f}\in \hat{\mathcal{B}}: \text{ $\hat{f}$ has compact support}\right\rbrace$.

\begin{theorem}\label{T:generalform}
Let $(X,\mathcal{X},\mu)$ be a measure space and $(\mathcal{E},\mathcal{F})$ a Dirichlet form on $L_2(X,\mu)$. Assume that $\mathcal{B}$ vanishes nowhere on $X$. Then the following hold:
\begin{enumerate}
\item[(i)] There exists a minimal energy dominant measure $m$ for the transferred Dirichlet form $(\hat{\mathcal{E}},\hat{\mathcal{F}})$ on the locally compact Hausdorff space $\Delta$. 
\item[(ii)] The form $(\hat{\mathcal{E}},\hat{\mathcal{B}}_c)$ is closable in $L_2(\Delta,m)$, and its closure $(\hat{\mathcal{E}},\hat{\mathcal{F}}^{(m)})$ is a regular symmetric Dirichlet form.
\item[(iii)] The Dirichlet form $(\hat{\mathcal{E}},\hat{\mathcal{F}}^{(m)})$ admits a carr\'e du champ.
\end{enumerate}
\end{theorem}

Theorem \ref{T:generalform} says that any energy form can be transferred to a possibly larger state space on which it admits a carr\'e du champ (with respect to a suitable reference measure).

\begin{remark}
\cite[Example 6.1]{HKT} tells that in general we cannot expect to be able to pull back the energy measures of $(\hat{\mathcal{E}},\hat{\mathcal{F}})$ on $\Delta$ to $X$. The same happens with the measure $m$, which is well defined on $\Delta$, but may have no pull-back to $X$.
\end{remark}

\end{document}